\documentclass[12pt,oneside]{amsart}
\usepackage{amscd,amsmath,amssymb,amsfonts,a4}
\usepackage[cmtip, all]{xy}
\usepackage{pstricks}
\usepackage{color}
\usepackage{hyperref}
\usepackage{enumerate}

\usepackage{colonequals}
\newlength{\rulebreite}

\def\timesover#1#2#3{\ \xymatrix@1@=0pt@M=0pt{ _{#1}&\times&_{#2} \\& ^{#3}&}\ }
\def\otimesover#1#2#3{\ \xymatrix@1@=0pt@M=0pt{ _{#1}&\otimes&_{#2} \\& ^{#3}&}\ }

\usepackage[all]{xy}
\theoremstyle{definition}

\newtheorem{thm}{Theorem}
\newtheorem{lem}[thm]{Lemma}
\newtheorem{cor}[thm]{Corollary}
\newtheorem{prop}[thm]{Proposition}
\theoremstyle{definition}
\newtheorem{defn}[thm]{Definition}

\newtheorem{rmk}[thm]{Remark}

\numberwithin{thm}{section}
\numberwithin{equation}{section}

\newcommand{\Spec}{{\rm Spec \,}}

\newcommand{\C}{{\mathbb C}}

\renewcommand{\P}{{\mathbb P}}
\newcommand{\Q}{{\mathbb Q}}
\newcommand{\R}{{\mathbb R}}

\newcommand{\Z}{{\mathbb Z}}
\newcommand{\g}{\mathfrak{g}}

\newcommand{\gl}{\mathfrak{gl}}
\newcommand{\la}{\mathfrak{a}}

\def\tilde{\widetilde}

\usepackage{fullpage}
\usepackage{microtype}
\begin{document}

\title{Algebraic entropy for smooth projective varieties}
\author{K. V. Shuddhodan}
\email{kvshud@purdue.edu}
\address{Department of Mathematics,
Purdue University, West Lafayette, IN 47907, USA}

\begin{abstract}

We show that the spectral radius for the action of a self map $f$ of a smooth projective variety (over an arbitrary base field) on its $\ell$-adic cohomology is achieved on the $f^*$ stable sub-algebra generated by any ample class.~This generalizes a result of Esnault-Srinivas who had obtained an analogous result for automorphisms of surfaces.~Over $\C$ we also show that this sub-algebra is naturally an irreducible representation of a Looijenga-Lunts-Verbitsky type Lie algebra acting on the cohomology of a smooth projective variety.
\end{abstract}


\maketitle

\section{Introduction}

Let $X$ be a compact K\"ahler manifold and $\omega \in H^{2}(X,\R)$ be a $(1,1)$ form corresponding to the choice of a K\"ahler metric.~Let $f \colon X \to X$ be a surjective and holomorphic self map of $X$.~To any such pair $(X,f)$ consisting of a compact metric space and a continuous self map one can associate a real number $d_{\text{top}}(f)$,~the topological entropy of the pair $(X,f)$ \cite{Bowen,Dinaburg}.

Let $\lambda(f),~\lambda_{\text{even}}(f)$ and $\lambda_{p}(f),~0 \leq p \leq \text{dim}(X)$ be the spectral radius of $f^*$ acting on $\oplus_{i} H^{i}(X,\Q)$,~$\oplus_{i} H^{2i}(X,\Q)$ and $H^{p,p}(X,\R)$ respectively.~The following fundamental theorem is due to Gromov-Yomdin.

\begin{thm}\cite{Gromov,Yomdin}\label{Gromov_Yomdin} With notations as above,~$d_{\text{top}}(f)=\log \lambda(f)=\log \lambda_{\text{even}}(f)=\max \limits_{0 \leq p \leq \text{dim}(X)} \log \lambda_p(f)$.

\end{thm}

Theorem \ref{Gromov_Yomdin} implies that for a surjective self map $f$ of a smooth projective variety over $\C$,~the spectral radius of $f^*$ on the Hodge classes equals the spectral radius on the entire cohomology (see \cite{Oguiso_ICM} for a comprehensive summary of the Gromov-Yomdin theory and its generalizations).

When working over an arbitrary base field there is no obvious and useful notion of a topological entropy,~however it still makes sense to look at the action of $f^*$ on suitable cohomology theories.~In this direction Esnault-Srinivas obtained the following result for automorphisms of smooth projective surfaces over an arbitrary base field.

\begin{thm}\cite[Theorem 1.1]{Esnault-Srinivas}\label{Esnault-Srinivas}
Let $f \colon X \to X$ be an automorphism of a smooth projective surface over an arbitrary algebraically closed field $k$.~Let $\ell$ be a prime invertible in $k$.~Let $\omega \in H^2(X,\Q_{\ell})$ be an ample class.~Then for any embedding of $\Q_{\ell}$ inside $\C$,

\begin{enumerate}

\item the spectral radius for the action of $f^*$ on $H^*(X,\Q_{\ell})$ coincides with the spectral radius for its action on the sub-space spanned by $f^{n*}\omega,~n \in \Z$.

\item Let $V(f,\omega)$ be the largest $f^*$-stable sub-space of $H^2(X,\Q_{\ell})$ in the orthogonal complement of $\omega$ (with respect to the cup-product pairing).~Then $f^*$ is of finite order on $V(f,\omega)$.

\end{enumerate}

\end{thm}

The proof of Theorem \ref{Esnault-Srinivas} is quite delicate,~and uses (among other things) the classification of smooth projective surfaces in positive characteristics.~It also relies on lifting of certain $K3$ surfaces to characteristic $0$ based on \cite{Maulik-Lieblich},~and uses Hodge theory to resolve this case.~Given the motivic nature of Theorem \ref{Esnault-Srinivas},~it is natural to ask for analogues of the Gromov-Yomdin theory over an arbitrary base field (see \cite{Esnault-Srinivas},~Section 6.2).~Indeed one has the following result.

\begin{thm}\cite[Corollary 1.2]{Shu19}\label{Shu_self_map}
Let $f \colon X \to X$ be any self map\footnote{In what follows by self map we will mean an endomorphism of a scheme.} of a proper scheme over an arbitrary field $k$.~Let $\ell$ be a prime invertible in $k$ and let $\bar{k}$ be an algebraic closure of $k$.~Then (for any embedding of $\Q_{\ell}$ in $\C$) the spectral radius of $f^*_{\bar{k}}$ on the entire $\ell$-adic cohomology equals the spectral radius for its action on $\oplus_{i}H^{2i}(X_{\bar{k}},\Q_{\ell})$.
\end{thm}
  
The proof Theorem \ref{Shu_self_map} uses the theory of weights \cite{Del80} to obtain restrictions on the analytic properties of a zeta function associated to the self map $f$ \cite[Definition 2.12]{Shu19}.~The analytic properties of the zeta functions are then used to obtain restrictions on the behavior of the spectral radius with respect to the weight filtration \cite[Theorem 1.1]{Shu19},~which in turn implies Theorem \ref{Shu_self_map}.

Now suppose $X$ is a smooth projective variety,~then one can ask more refined questions and in particular look for analogues of Theorem \ref{Esnault-Srinivas},~(1) for higher dimensional varieties.~In this direction we have the following result.
 
\begin{thm}[\cite{Tuyen_Weil_conjectures}]\label{Truong}

Let $f \colon X \to X$ be a dominant self map of a smooth projective variety over an algebraically closed field $k$.~Let $\ell$ be a prime invertible in $k$.~Then for any embedding of $\Q_{\ell}$ inside $\C$,~the spectral radius of $f^*$ on the $\ell$-adic cohomology of $X$ equals the spectral radius of $f^*$ acting on its Chow group modulo numerical equivalence.

\end{thm}

Motivated by the methods in \cite{Tuyen_Weil_conjectures} and in an forthcoming article \cite[Appendix B]{SV20} we obtain the following generalization of Theorem \ref{Esnault-Srinivas} to higher dimensions.

\begin{thm}\label{main_theorem_introduction}

Let $f \colon X \to X$ be a self map of a smooth projective variety over an arbitrary algebraically closed field $k$.~Let $\ell$ be a prime invertible in $k$.~Let $[\omega] \in H^2(X,\Q_{\ell})$ be an ample class.~Then the spectral radius of $f^*$ acting on $H^*(X,\Q_{\ell})$ (with respect to $\tau:\Q_{\ell} \hookrightarrow \C$) is independent of $\tau$,~and coincides with the spectral radius of $f^*$ on the numerical Gromov algebra (see Definition \ref{homological_Gromov}).

\end{thm}

Theorem \ref{main_theorem_introduction} in particular shows that for smooth projective varieties with $\text{Pic}(X)=\Z$,~the spectral radius of a self map on the $\ell$-adic cohomology coincides with its degree (see Corollary \ref{correspondence_variety_Serre}).~Finally in a hope to generalize Theorem \ref{Esnault-Srinivas},~(2),~we propose an approach via a Looijenga-Lunts-Verbistky Lie algebra \cite{LL97,Ver96}.~The reader is referred to Section \ref{LLV} for details. 

\subsubsection*{Acknowledgements:}

I would like to thank Prof.~Vasudevan Srinivas for helpful comments and suggestions,~and in particular for motivating me to study the Gromov algebra even before Theorem \ref{main_theorem_introduction} came to existence.~I would also like to thank Radhika Ganapathy for numerous discussions and clarifications regarding LLV Lie algebra.~Finally I am thankful to the referee for a careful reading of the article and the many suggestions which have improved the exposition and readability of the article.
  
\section{Some preliminaries from intersection theory}\label{prelimnaries_intersection_theory}

Throughout this article we will work over an arbitrary algebraically closed field $k$.~ A variety (over $k$) is a finite type,~separated and integral scheme over $k$.~Let $\ell$ be a prime invertible in $k$.~We fix once and for all an isomorphism of $\Q_{\ell}(1)$ with $\Q_{\ell}$.~Hence we will talk of cycles classes with values in $\ell$-adic cohomology without the Tate twist.~We also fix an embedding 
 
 \begin{equation}\label{embedding_chapter_four}
 \tau \colon \Q_{\ell} \hookrightarrow \C.
 \end{equation}

\subsection{Summary of results needed from intersection theory} Let $X$ be a smooth,~projective variety over $k$ of dimension $r$.~Let $Z^*(X)$ be the free abelian group generated by the set of closed subvarieties of $X$ and graded by codimension \cite[Section 1.3]{Fulton_Intersection_Theory}.~Let $A^{*}(X)$ be the graded (by codimension) Chow ring of $X$ \cite[Section 8.3]{Fulton_Intersection_Theory}.~The group underlying $A^*(X)$ is a graded quotient of $Z^*(X)$ by rational equivalence.~We shall write $A(X) \colonequals \oplus_{i}A^{i}(X)$ when we want to ignore the grading and the ring structure.

The \textit{components} of an algebraic cycle $[Z] \in Z^{*}(X)$ are the subvarieties of $X$ which appear in $[Z]$ with non-zero coefficients.~To any closed \textit{subscheme} $Y \subseteq X$ we can associate an effective cycle $[Y]$ in $Z^*(X)$ whose components are precisely the irreducible components of $Y$ \cite[Section 1.5]{Fulton_Intersection_Theory}.

 Let $A_{\text{num}}^*(X)$ (resp. $A_{\text{num}}^*(X)_{\Q}$,~resp.  $A_{\text{num}}^*(X)_{\R}$) be the graded (by codimension) ring of algebraic cycles on $X$ modulo numerical equivalence with $\Z$ (resp. $\Q$,~resp. $\R$) coefficients \cite[Section 1.1]{Kleiman_algebraic_cycles}.
 
 Let $A_{\text{hom}}^*(X)_{\Q}$ be the graded (by codimension) ring of algebraic cycles on $X$ modulo homological equivalence (with respect to $\ell$-adic cohomology \footnote{Conjecturally $A_{\text{hom}}^*(X)_{\Q}$ is independent of $\ell$} ),~with $\Q$ coefficients (see \cite[Chapitre 4]{SGA 4.5} and \cite[Chapter 6]{Milne_Etale_Cohomology} for a construction of cycle classes).~Note that $A_{\text{num}}^*(X)_{\Q}$ is a quotient of $A_{\text{hom}}^*(X)_{\Q}$,~which in turn is a $\Q$-subalgebra of $\oplus_{i}H^{2i}(X,\Q_{\ell})$.
 
 For a morphism $f \colon X \to Y$ of smooth,~projective varieties over $k$,~there is a pullback map $f^* \colon A^{*}(Y) \to A^*(X)$ and a pushforward map $f_* \colon A(X) \to A(Y)$ \cite[Proposition 8.3 (a) and Theorem 1.4]{Fulton_Intersection_Theory}.~The pullback is a morphism of graded rings and the pushforward is a morphism of abelian groups.~Further they satisfy a projection formula \cite[Proposition 8.3 (c)]{Fulton_Intersection_Theory}.~In particular there exists a group homomorphism $\pi_{X*} \colon A(X) \to A(\Spec(k)) \simeq \Z \cdot [\Spec(k)]$ \cite[Definition 1.4]{Fulton_Intersection_Theory}.~ $A_{\text{num}}^*(X)$ and $A_{\text{hom}}^*(X)_{\Q}$ also have similar functorial properties \cite[Section 1]{Kleiman_algebraic_cycles}.~  We shall denote the intersection product on rings by `$\cdot$'.~For cycles $[Z]$ and $[Z']$ of complimentary co-dimension in $X$,~by abuse of notation we shall also denote the integer $\pi_{X*}([Z] \cdot [Z'])$ by $[Z] \cdot [Z']$.
 
 
Let $[\P_k^{s}] \in A^{n-s}(\P^n_k),~0 \leq s \leq n$ be the class of a $s$-dimensional linear subspace of $\P_k^n$.~The Chow ring $A^*(\P^n_k)$ is isomorphic to the graded ring  $\Z[x]/(x^{n+1})$ under the map $[\P^{n-1}] \to x$ \cite[Proposition 8.4]{Fulton_Intersection_Theory} and the class $[\P_k^{s}]$ generates the free abelian group $A^{n-s}(\P^n_k),~0 \leq s \leq n$ \cite[Example 1.9.3]{Fulton_Intersection_Theory}.

\begin{defn}

The \textit{degree} of $[Z] \in A^{s}(\P^{n}_k)$ is the integer $[Z] \cdot [\P^{s}_k]$.~For a subvariety $Z $ of $\P^n_k$ by $\text{deg}(Z)$ we mean $\text{deg}([Z])$.

\end{defn}

 
For any two smooth,~projective varieties $X$ and $Y$ (over $k$),~there is an exterior product map $A^*(X) \otimes_{\Z} A^{*}(Y) \to A^*(X \times_k Y)$ \cite[Section 1.10]{Fulton_Intersection_Theory},~which is a morphism of graded rings \cite[Example 8.3.7]{Fulton_Intersection_Theory}.~We shall denote the image of $[Z] \otimes [Z']$ by $[Z] \times [Z']$.

In what follows,~we will need a bound (see Proposition \ref{basic_bound_intersection_product}) well known to experts and proved using standard techniques.~For ease of exposition we present a short proof using Chow's moving Lemma and the join construction \cite[Example 8.4.6]{Fulton_Intersection_Theory}.
\begin{defn}\label{proper_intersection}

Two subvarieties $V$ and $W$  in a smooth projective variety $X$  are said to \textit{intersect properly},~if the each component of $V \cap W$ has the right dimension (i.e. $\text{dim}(V)+\text{dim}(W)-\text{dim}(X)$).

\end{defn}

\begin{rmk}\label{cycles_intersect_properly}

In a similar vein,~cycles $[V]$ and $[W]$ in $Z^*(X)$ are said to \textit{intersect properly} if each component of $[V]$ intersects each component of $[W]$ properly.

\end{rmk}

Suppose now $i \colon X \hookrightarrow \P^n_k$ is a closed embedding of a smooth,~projective variety of dimension $r$.~Fulton's definition of intersection multiplicities implies the following statement \cite[Section 6.2, Section 7.1]{Fulton_Intersection_Theory}.

\begin{prop}\label{intersection_formula_proper}

Let $[C] \in Z^*(\P^n_k)$ be a cycle on $\P_k^n$ which intersects $[X]$ properly.~Then,

\begin{center}
$i^{*}([C])=\sum_j i(Z_j;[X],[C])[Z_j] \in A^*(X)$,
\end{center}

\noindent where $Z_j$'s are the irreducible components of the intersection of $X$ with the components of $[C]$,~and $i(Z_j;[X],[C])$'s are the intersection multiplicities along the $Z_j$'s \cite[Definition 7.1]{Fulton_Intersection_Theory}.

\end{prop}

\begin{rmk}\label{intersection_notation}

By abuse of notation the cycle $\sum_j i(Z_j;X,C)[Z_j] \in Z^*(X)$ will also be denoted by $[C] \cdot [X]$.~Moreover if $[C]$ is an effective cycle so is $[C] \cdot [X]$ \cite[Proposition 7.1]{Fulton_Intersection_Theory}.

\end{rmk}

Let $V \subseteq X$ be a closed subvariety of dimension $d$.~Let $L \subseteq \P_k^n$ be a linear subspace of dimension $n-r-1$ disjoint from $X$.~We denote by $C_L(V) \subseteq \P^n_k$,~the cone of $V$ over $L$ \cite[Section 2]{Roberts} or equivalently the join of $V$ and $L$ \cite[Example 8.4.5]{Fulton_Intersection_Theory}.~It is a subvariety of dimension $n+d-r$,~and of degree equal to the degree of $V$.~Moreover $V$ is an irreducible component of $C_L(V) \cap X$ and every component of $C_L(V) \cap X$ is of dimension equal to $d$ \cite[Lemma 2]{Roberts}.

\begin{rmk}\label{cone_intersection}

Hence for any such $L$,~we see that $C_L(V)$ and $X$ intersect properly (see Defintion \ref{proper_intersection}) and we let $[C_L(V)] \cdot [X]$ denotes the corresponding cycle on $X$ (see Remark \ref{intersection_notation}).

\end{rmk}

For an arbitrary cycle $[V]=\sum_i m_i[V_i] \in Z^{r-d}(X)$ we define

\begin{center}
 $[C_L([V])]:=\sum_i m_i [C_L(V_i)] \in Z^{r-d}(\P^n)$.
 \end{center}

Let $V$ and $W$ be closed subvarieties of $X$.~We define the \textit{excess} of $V$ (relative to $W$) to be $-\infty$ if they do not intersect.~Else it is defined to be the maximum of the (non-negative) integers

\begin{center}

$\text{dim}(Y)-\text{dim}(V)-\text{dim}(W)+\text{dim}(X)$,

\end{center}

\noindent where $Y$ runs through all the components of $V \cap W$.~We denote the excess by $e_W(V)$.~For a cycle $[V] \colonequals \sum_i m_i[V_i]$ in $Z^*(X)$,~we define $e_W([V]) \colonequals \sum_i m_i e_W(V_i)$.~We have the following result from \cite{Roberts} (used there to prove the `Chow moving Lemma').

\begin{lem}\label{main_lemma_Roberts}\cite[Main Lemma]{Roberts}

Let $i \colon X \hookrightarrow \P^n_k$ be a smooth,~projective closed subvariety of dimension $r$.~Let $W$ be a subvariety of $X$.~For any cycle $[V] \in Z^*(X)$,~there exists a dense open subset $U$ of $G(n,n-r-1)$,~the Grassmanian of linear subspaces in $\P^n$ of dimension $n-r-1$,~such that for any closed point $x \in U$,~if $L_x$ denotes the corresponding linear subspace then,

\begin{enumerate}

\item[(1)] $L_x \cap X =\emptyset$.

\item[(2)] $e_W\left ([C_L \left ([V] \right )] \cdot [X]-[V] \right ) \leq \text{max}\left ( e_W  ([V])-1,0 \right ) $.

\end{enumerate}

\end{lem}

\subsection{An estimate for intersection product}

As before let $i \colon X \hookrightarrow \P^n_k$ be a smooth,~projective closed subvariety of dimension $r$.~Let $V$ and $W$ be closed subvarieties of $X$.~Let $d$ be the dimension of $V$.~The following lemma is now easy to deduce from Lemma \ref{main_lemma_Roberts}.

\begin{lem}\label{moving_lemma}

There exist a positive integer $k \leq r+1$ and a sequence of effective cycles $\{[V_j]\}_{0 \leq j \leq k}$ and $\{[E_j]\}_{1 \leq j \leq k}$ in $Z^{r-d}(X)$ such that,

\begin{enumerate}

\item[(1)] $[V_0]=[V]$ in $Z^{r-d}(X)$.

\item[(2)] $[V_j]=[E_{j+1}]-[V_{j+1}]$ in $Z^{r-d}(X)$ for all $0 \leq j \leq k-1$.

\item[(3)] For all $j \geq 1$,~the $[E_j]$'s are `ambient' cycles that is,~$[E_j]=i^{*}\left ( \text{deg} \left ( [V_{j-1}] \right )[\P^{n-d+r}_{k}] \right )$ in $A^{r-d}(X)$.

\item[(4)] Every component of $[V_{k-1}]$ and $[V_k]$ intersects $W$ properly (see Definition \ref{proper_intersection} and Remark \ref{cycles_intersect_properly}).

\end{enumerate}

In particular 

\begin{center}
$[V]=\sum_{j=1}^{k}(-1)^{j+1}[E_j]+(-1)^{k}[V_{k}]$ in $Z^{r-d}(X)$.
\end{center}

\end{lem}

\begin{proof}

Let 

\begin{center}

$[V_0] \colonequals [V] \in Z^{r-d}(X)$.

\end{center}

For any integer $j \geq 1$,~having defined $[V_{j-1}] \in Z^{r-d}(X)$ and proven that it is effective,~we define 

\begin{equation}\label{define_E_j}
[E_j] \colonequals [C_{L_j}([V_{j-1}])] \cdot [X]  \in Z^{r-d}(X), 
\end{equation}

\noindent where $L_j$ is linear sub-space of $\P^n$ of dimension $n-r-1$ (see Remark \ref{cone_intersection}),~chosen such that 

\begin{center}
$e(i^*[C_{L_j}([V_{j-1}])]-[V_{j-1}]) \leq \text{max}(e([V_{j-1}]-)1,0)$ (see Lemma \ref{main_lemma_Roberts}).
\end{center}

Here the excess is with respect to $W$.~Since $[C_{L_j}(V_{j-1})]$ and $[X]$ intersect properly \cite[Lemma 2]{Roberts},~Remark \ref{intersection_notation} implies that $[E_j]$ is an effective cycle.~For any integer $j$ having defined $[V_{j-1}]$ and $[E_{j}]$,~we define,

\begin{center}

$[V_{j}] \colonequals [E_j]-[V_{j-1}]$ in $Z^{r-d}(X)$.

\end{center}

For any subvariety $V \subseteq X$,~$V$ is an irreducible component of $C_L(V) \cap X$ \cite[Lemma 2]{Roberts},~thus the effectivity of $[V_{j}]$ for any $j \geq 1$,~is a consequence of the effectivity of $[E_j]$.

Since $e([V_0])=e([V]) \leq r$,~for any $j \geq r$,~the excess $e([V_j])=0$.~Let $k-1$ be the smallest integer $j$ with the property that $e([V_{k-1}])=0$.~Then every component of the algebraic cycles $[V_{k-1}]$  and $[V_k]$ intersects $W$ properly. 

For any $j\geq 1$ since $C_L([V_{j-1}])$ and $X$ intersect properly \cite[Lemma 2]{Roberts},~Proposition \ref{intersection_formula_proper} implies that 

\begin{equation}\label{first_step_lemma}
[E_j]=i^{*}\left ([C_L([V_{j-1}])] \right) \in A^{r-d}(X).
\end{equation}

For any $j \geq 1$ since $[C_L([V_{j-1}])]$ as a cycle on $\P_k^n$ has degree equal to the degree of $[V_{j-1}]$ \cite[Example 8.4.5]{Fulton_Intersection_Theory},~thus (\ref{first_step_lemma}) implies that,

\begin{center}\label{second_step_lemma}
$[E_j]=i^{*}\left ( \text{deg} \left ( [V_{j-1}] \right )[\P^{n-d+r}_{k}] \right )$ in $A^{r-d}(X)$.
\end{center}

\end{proof}

Now we derive a basic estimate which is needed later.

\begin{prop}\label{basic_bound_intersection_product}

Let $X \subseteq \P^n_k$ be a smooth,~projective variety.~Then for any two subvarieties $V$ and $W$ of complimentary dimension in $X$,~$|[V] \cdot [W]| \leq  C\text{deg}(V)\text{deg}(W)$,~for a constant $C$ independent of $V$ and $W$ .


\end{prop}

\begin{proof}

%

We use Lemma \ref{moving_lemma} to construct a sequence of algebraic cycles $\{[V_j]\}_{0 \leq j \leq k}$ and $\{[E_j]\}_{1 \leq j \leq k}$ in $Z^{r-d}(X)$ where $d$ is the co-dimension of $V$ in $X$ and satisfying properties (1)-(4) in Lemma \ref{moving_lemma}.

Since

\begin{center}
$[V]=\sum_{j=1}^{k}(-1)^{j+1}[E_j]+(-1)^{k}[V_{k}]$ in $Z^{r-d}(X)$,
\end{center}

\noindent one has that

\begin{equation}\label{basic_bound_first}
|[V] \cdot [W]| \leq \sum_{j=1}^{k}|[E_j] \cdot [W]|+|[V_k] \cdot [W]|. 
\end{equation}

Note that $[E_j]=i^{*} \left ( \text{deg}([V_{j-1}])[\P^{n+d-r}_{k}] \right )$ (see Lemma \ref{moving_lemma},~(3)) and hence for every $j \geq 1$,~

\begin{equation}\label{basic_bound_second}
[E_j] \cdot [W]=\text{deg}(W)\text{deg}([V_{j-1}]).
\end{equation} 
 
Since every component of $[V_{k-1}]$ intersects $[W]$ properly,~$[V_k] \cdot [W]$ is bounded above by $[E_k] \cdot [W]=\text{deg}(W)\text{deg}([V_{k-1}])$ \cite[Proposition 7.1]{Fulton_Intersection_Theory}.~Combining (\ref{basic_bound_first}) and (\ref{basic_bound_second}) we get,

\begin{equation}\label{basic_bound_third}
|[V] \cdot [W]| \leq \left ( \sum_{j=1}^{k}\text{deg}([V_{j-1}])+\text{deg}([V_{k-1}]) \right) \text{deg}(W).
\end{equation}

Projection formula implies that for every $j \geq 1$,

\begin{center}
$\text{deg}([E_j])=\text{deg}(X) \text{deg}([V_{j-1}])$.
\end{center}

Since the $[E_j]$'s and $[V_j]$'s are effective,~

\begin{center}
$\text{deg}([V_{j}]) \leq \text{deg}([E_j])=\text{deg}(X)\text{deg}([V_{j-1}])$.
\end{center}

Thus for every $j \geq 1$ 

\begin{equation}\label{basic_bound_fourth}
\text{deg}([V_{j}]) \leq \text{deg}(X)^{j}\text{deg}(V) \leq \text{deg}(X)^{r+1}\text{deg}(V).
\end{equation}

Thus (\ref{basic_bound_third}) and (\ref{basic_bound_fourth}) together imply that

\begin{equation}\label{final_bound_uniform}
 |[V] \cdot [W]| \leq (r+2)\text{deg}(X)^{r+1} \text{deg}(V)\text{deg}(W).
 \end{equation}
 
\end{proof}

\begin{rmk}
If $V$ and $W$ intersect properly, then the bound in (\ref{final_bound_uniform}) can be improved to  $[V] \cdot [W] \leq \text{deg}(V)\text{deg}(W)$ \cite[Lemma 10.12]{Hru12}.~Furthermore the bound in Proposition \ref{basic_bound_intersection_product} can be generalized (in an appropriate sense) to Gysin pullbacks under regular embeddings of quasi-projective varieties and proved without recourse to the moving lemma \cite[Appendix B]{SV20}.
\end{rmk}

\section{Gromov algebra}\label{Gromov_sub-algebra}

Let $i:X \hookrightarrow \P^n_k$ be a smooth,~projective variety over an algebraically closed field $k$ of dimension $r$.~Let $[H] \in A^1(X)$ be the class of a hyperplane section.~Let $\omega$ be the cohomology class of $[H]$ in $H^2(X,\Q_{\ell})$.

For $j \geq 1$,~let $[H]^{j}$ (resp. $\omega^j$) denote the $j^{\mathrm{th}}$ self intersection product (resp. self cup product) of $[H]$ (resp. $\omega$) in $A^*(X)$ (resp. $H^*(X,\Q_{\ell})$).~Let $f \colon X \to X$ be a self map and $[\Gamma_f] \in A^{r}(X \times_k X)$ be the graph correspondence.~For integers $0 \leq j \leq r$ let 

\begin{equation}\label{intersection_number}
\delta_j(f):=[H]^{r-j} \cdot f^{*}([H]^j)= f^{*}([H]^j) \cdot [H]^{r-j}.
\end{equation}

Note that we have an equality

\begin{equation}\label{intersection_number_cup_product}
\delta_j(f)=\text{Tr}_{X}(\omega^{r-j} \cup f^*(\omega^j))=\text{Tr}_X( f^*(\omega^j) \cup \omega^{r-j}),
\end{equation}

\noindent where $\cup$ is the cup product on $H^{*}(X,\Q_{\ell})$ and $\text{Tr}_X$ is the trace map $\text{Tr}_X:H^{2r}(X,\Q_{\ell}) \to \Q_{\ell}$.

%
%
%
%
%
%
%
\begin{lem}

Using the above notations,

\begin{equation}\label{class_graph}
 (i \times i)_*[\Gamma_f]=\sum_{j=0}^{r} \delta_{r-j}(f)([\P_k^{r-j}] \times [\P_k^j])
\end{equation}

\end{lem}

\begin{proof}


The exterior product map $A^*(\P^n_k) \otimes_{\Z} A^{*} (\P^n_k) \to A^*(\P^n_k \times_k \P^n_k)$ is an isomorphism of graded rings \cite[Example 8.3]{Fulton_Intersection_Theory} and hence 

\begin{equation}\label{class_of_Y}
(i \times i)_*[\Gamma_f]=\sum_{j=0}^{r} n_j ( [\P_k^{r-j}] \times [\P_k^j])
\end{equation}

\noindent where for any $j \geq 0$,~$n_j=\left ( \left (i \times i \right )_*[\Gamma_f] \cdot \left ( [\P^{n-r+j}_k] \times [\P^{n-j}_k] \right ) \right )$.~The projection formula thus implies that

\begin{equation}\label{class_of_Y_II}
n_j= [\Gamma_f] \cdot \left ([H]^{r-j} \times [H]^j \right),
\end{equation}

\noindent and the result follows from Lefschetz trace formula \cite[Section 3.3.3]{Yves_Andre}.

\end{proof}

%
%

\begin{defn}\label{homological_Gromov}

The \textit{homological Gromov algebra} $\text{A}^{Gr}_{\text{hom}}(f,\omega)_{\Q}$ \footnote{In what follows when $k=\C$,~we use Betti cohomology instead of $\ell$-adic cohomology to define homological equivalence and hence to define $\text{A}^{Gr}_{\text{hom}}(f,\omega)_{\Q}$.~Thus when $k=\C$ there is no dependence on an auxillary prime $\ell$.} (resp. the \textit{numerical Gromov algebra}, $\text{A}^{Gr}_{\text{num}}(f,[H])_{\Q}$)   is the smallest $f^*$-stable (unital) sub-algebra of $A^{*}_{\text{hom}}(X)_{\Q}$  (resp. $A^{*}_{\text{num}}(X)_{\Q}$) containing $\omega$ (resp. $[H]$).

\end{defn}



%
%
%


%
%

The numerical Gromov algebra with real coefficients $\text{A}^{Gr}_{\text{num}}(f,[H])_{\R}$ is the $\R$-algebra $\text{A}^{Gr}_{\text{num}}(f,[H])_{\Q} \otimes_\Q \R$.

Let $\lambda_i$ (resp. $\chi_i$) be the spectral radius \footnote{In what follows the spectral radius of a linear endomorphism of a vector space over a sub field of $\C$ is defined to be the maximum in absolute value of its complex eigenvalues.} of $f^*$ acting on $A_{\text{num}}^i(X)_{\Q},~0 \leq i \leq \text{dim}(X)$ (resp. $A_{\text{hom}}^i(X)_{\Q},~0 \leq i \leq \text{dim}(X)$).~Let $\lambda^{Gr}$ and $\chi^{Gr}$ be the spectral radii of $f^*$ acting on $\text{A}^{Gr}_{\text{num}}(f,[H])_{\Q}$ and $\text{A}^{Gr}_{\text{hom}}(f,\omega)_{\Q}$ respectively.~Note that $\lambda^{Gr}$ is also the spectral radius of $f^*$ acting on $\text{A}^{Gr}_{\text{num}}(f,[H])_{\R}$.

Let $\mu_j$ be the spectral radius (with respect to $\tau$ in (\ref{embedding_chapter_four})) of $f^*$ acting on $H^{j}(X,\Q_{\ell}),~0 \leq j \leq 2 \, \text{dim}(X)$.~Following lemma is obvious.

\begin{lem}\label{Basic_inequality}

Using the above notations we have inequalities
\begin{center}
$\lambda^{Gr} \leq \max \limits_{ 0 \leq i \leq \text{dim}(X)} \, \lambda_{i} \leq \max \limits_{ 0 \leq i \leq \text{dim}(X)} \, \chi_{i} \leq \max \limits_{0 \leq j \leq \, 2\text{dim}(X)} \, \mu_j.$
\end{center}
Further 
\begin{center}
$\lambda^{\text{Gr}} \leq \chi^{Gr} \leq \max \limits_{0 \leq i \leq \text{dim}(X)} \, \chi_{i}$.
\end{center}

\end{lem}

We will also need the following lemma.

\begin{lem}\label{bound_lim_sup}

Let $\{a_{m,i}\}_{m \geq 1},1 \leq i \leq s$ be a collection of sequences of complex number.~Let $b_{i},~i \leq i \leq s$ be arbitrary complex numbers.~Then

\begin{center}

$\limsup \limits_{m} |\sum_{i=1}^{s} a_{m,i}b_{i}|^{1/m} \leq \max \limits_{1 \leq i \leq s} \limsup \limits_{m} |a_{m,i}|^{1/m}$.

\end{center}

\end{lem}

\begin{proof}\footnote{The short and elegant proof here was suggested to us by the referee.}
The radius of absolute convergence of the power series $\sum_{ m \geq 0} \sum_{i=1}^sb_ia_{m,i}z^m$ is at least as large as the minimum of the radius of absolute convergence of $\sum_{m \geq 0} a_{m,i}z^m$ as $1 \leq i \leq s$.~The desired result is now an immediate consequence of the Cauchy-Hadamard formula for the radius of absolute convergence.

%
%
%
%
%

\end{proof}

Let $V$ be any finite dimensional vector space over $\R$ (or $\C)$ and $T:V \to V$ a linear map.~Let $|| \cdot ||$ be any matrix norm.~In what follows we will make use of the following theorem due to Gelfand \cite[Theorem 18.9]{Rudin_Real_Complex_Analysis}.

\begin{thm}\label{Gelfand_Spectral_radius}

$\limsup \limits_{m} ||T^{m}||^{1/m}=\rho(T)$,~where $\rho(T)$ is the spectral radius of $T$.

\end{thm}

%
%
%
%
%
%
%
%
%
%
%

Let $K$ be a normed field such that there exists an embedding $\tau:K \hookrightarrow \C$ of normed fields.~Let $V$ be any finite dimensional vector space over $K$ and $T:V \to V$ a linear map.~We shall need the following standard result which we state without a proof.

\begin{prop}\label{Trace_Spectral_radius}

$\limsup \limits_{m}|\text{Tr}(T^m)|^{1/m}=\rho(T)$,~where $\rho(T)$ is the spectral radius of $T$.

\end{prop}

%
%
%
%
%
%
%
%
%
%
%
%
%
%
%
%
%

Now we prove the principal result of this article.

\begin{thm}\label{Gromov_arbitrary_base_field}

Let $X$ be a smooth, projective variety over an arbitrary algebraically closed field $k$ of dimension $r$.~Let $[H] \in A^{1}(X)$ (respectively $\omega \in H^{2}(X,\Q_{\ell})$) be the class of an hyperplane section in the Chow group (respectively $\ell$-adic cohomology).~Let $f \colon X \to X$ be a self map of $X$.~Then all the inequalities in Lemma \ref{Basic_inequality} are in fact equalities.

Thus the spectral radius of $f^*$ acting on $H^*(X,\Q_{\ell})$ (with respect to $\tau:\Q_{\ell} \hookrightarrow \C$) is independent of $\tau$,~and coincides with the spectral radius of $f^*$ on the numerical Gromov algebra.

\end{thm}

\begin{proof}
 
Let $i \colon X \hookrightarrow \P^n_k$ be an embedding with $[H]$ being the class of the hyperplane section under $i$.~Clearly it suffices to show that

\begin{equation}\label{main_inequality_needed}
\lambda^{\text{Gr}} \geq \mu_i,~0 \leq i \leq 2r.
\end{equation}

Let $[\Gamma_{f^m}] \in A^{r}(X \times_k X)$ be the graph of the $m^{\mathrm{th}}$ iterate of $f$.~As before we denote by $\delta_j(f^{m})=[H]^{r-j} \cdot (f^{m*}[H]^j)=(f^{m*}[H]^j) \cdot [H]^{r-j}$ (see Definition (\ref{intersection_number})).~We shall first show that for any integer $i \in [0,2r]$,

\begin{equation}\label{need_to_Show}
\mu_i \leq \max \limits_{0 \leq j \leq r} \limsup \limits_m |\delta_j(f^m)|^{1/m}.
\end{equation}

It is clear from the definition of $\mu_i$ and $\delta_j(f^m)$ (see Equation \ref{intersection_number_cup_product}) that they specialise well,~and thus it suffices to prove the bound (\ref{need_to_Show}),~when $k$ is an algebraic closure of a finite field.~Note that this is not the case with the bound in (\ref{main_inequality_needed}).~Hence we need this intermediate step.~Hence we now assume that $k$ is an algebraic closure of a finite field,~and that $X$ and the self map $f$ are defined over this finite field.


The work of Katz-Messing \cite[Theorem 2.1]{Katz-Messing} and the Lefschetz trace formula  imply that for every integer $i \in [0,2r]$,~there exist an algebraic cycle $\pi^{i}_X \in Z^{r}(X \times X)_{\Q}$ (the $i^{\mathrm{th}}$  `Kunneth component',~see \cite[Section 3.3.3]{Yves_Andre}) such that

\begin{equation}\label{trace_kunneth}
\text{Tr}(f^{m*};H^{i}(X,\Q_{\ell}))=(-1)^{i}[\Gamma_{f^{m}}] \cdot \pi^{2r-i}_X,~
\end{equation}

\noindent representing the trace as an intersection product (on the product variety $X \times_k X$).~Recall that we have fixed an embedding $\tau:\Q_{\ell} \hookrightarrow \C$ (see (\ref{embedding_chapter_four})).~Thus $\Q_{\ell}$ is a normed field via this embedding.~Proposition \ref{Trace_Spectral_radius} and  (\ref{trace_kunneth}) together imply that,

\begin{equation}\label{first_bound}
\mu_i = \limsup_m |[\Gamma_{f^{m}}] \cdot \pi^{2r-i}_X|^{1/m},~0 \leq i \leq 2r.
\end{equation}

There exist finitely many  subvarieties $W_{j}^{2r-i} \subseteq X \times_k X$ of codimension $r$ (the components of the `Kunneth components') and a constant $C'$ such that for every $m \geq 1$,

\begin{equation}\label{second_bound}
|[\Gamma_{f^{m}}] \cdot \pi^{2r-i}_X| \leq C'\sum_{j} |[\Gamma_{f^m}] \cdot [W_j^{2r-i}]|,~0 \leq i \leq 2r.
\end{equation}

The estimate in Proposition \ref{basic_bound_intersection_product} (applied to the smooth projective variety $X \times_k X \subseteq \P^{n^2+2n}_k$) and (\ref{second_bound}) imply that there exists a constant $C''$ (depending only on $i:X \hookrightarrow \P^n_k$ and the choice of Kunneth components) such that for every $m \geq 1$,

\begin{equation}\label{third_bound}
|[\Gamma_{f^m}] \cdot \pi^{2r-i}_X| \leq C''\text{deg}([\Gamma_{f^m}])(\sum_{j} \text{deg}(W_{j}^{2r-1})),~0 \leq i \leq 2r.
\end{equation}

The degree in (\ref{third_bound}) is with respect to the embedding $X \times_k X \subseteq \P^{n^2+2n}_k$.~Moreover Lemma \ref{class_graph} implies that 

\begin{center}

$\text{deg}(\Gamma_{f^m})=\sum_{j=0}^r \delta_{r-j}(f^m) \text{deg}([\P_k^{r-j}] \times [\P^j_k])$.

\end{center}

Hence (\ref{first_bound}) and (\ref{third_bound}) together with Lemma \ref{bound_lim_sup} imply that for any integer $i \in [0,2r]$,~

\begin{equation}\label{fourth_bound}
\mu_i \leq \max \limits_{0 \leq j \leq r} \limsup_m |\delta_{j}(f^m)|^{1/m}.
\end{equation}

Thus we have obtained the bound (\ref{need_to_Show}) over an arbitrary algebraically closed field.

\vspace{0.5cm}

For the rest of the proof we work over the algebraically closed field $k$ we started with.~Let $\text{A}^{Gr}_{\text{num}}(f,[H])_{\R}$ be the numerical Gromov algebra with $\R$-coefficients.~Let $||.||$ be any norm on the finite dimensional $\R$-vector space $\text{A}^{Gr}_{\text{num}}(f,[H])_{\R}$.~Note that $f^*$ is a graded linear transformation of $\text{A}^{Gr}_{\text{num}}(f,[H])_{\R}$.~For every integer $m \geq 1$,~we denote the norm of the linear map $f^{m*}$ acting on $\text{A}^{Gr}_{\text{num}}(f,[H])_{\R}$ by $||f^{m*}||$.

Recall that $\delta_{j}(f^m)=f^{m*}([H]^j) \cdot [H]^{r-j}$.~Since the intersection product is bilinear,~the map from the $j^{\mathrm{th}}$ graded part of $\text{A}^{Gr}_{\text{num}}(f,[H])_{\R}$ to $\R$,~obtained by taking intersection product with $[H]^{r-j}$ is linear.~Consequently there exists a constant $\tilde{C'}$ independent of $m$,~such that for any $m \geq 1$,

\begin{equation}\label{sixth_bound}
|\delta_j(f^{m})| \leq \tilde{C'}||f^{m*}(H^{j})||,~0 \leq j \leq r.
\end{equation}

Since $f^*$ is a linear map,~(\ref{sixth_bound}) implies that there exists a constant $\tilde{C}$ independent of $m$ such that,~for any $m \geq 1$,~

\begin{equation}\label{fifth_bound}
|\delta_j(f^m)|^{1/m} \leq \tilde{C}^{1/m}||f^{m*}||^{1/m},~0 \leq j \leq r.
\end{equation}

Thus Theorem \ref{Gelfand_Spectral_radius},~(\ref{fourth_bound}) and (\ref{fifth_bound}) together imply 

\begin{center}

$\mu_i \leq \lambda^{Gr},~0 \leq i \leq 2r$.

\end{center}

\end{proof}

\begin{cor}\label{correspondence_variety_Serre}
Let $X$ be a smooth, projective variety over an arbitrary algebraically closed field $k$ of dimension $r$.~Let $f \colon X \to X$ be a self map  such that $f^*(\omega)=\lambda \omega$ for some ample class $\omega \in H^{2}(X,\Q_{\ell})$ and integer $\lambda$.~Then the spectral radius of $f^*$ on $H^*(X,\Q_{\ell})$ is $\lambda^r$.
\end{cor}

When $k$ is of characteristic $0$,~Corollary \ref{correspondence_variety_Serre} is a consequence of Serre's result \cite[Th\'eor\`eme 1]{Ser60}.~Also note that the assumption of Corollary \ref{correspondence_variety_Serre} is automatic when $\text{Pic}(X)=\Z$.

\section{Gromov algebra as a Lefschetz Module}\label{LLV}

The aim of this section is to give a representation theoretic perspective to the (homological) Gromov algebra by relating it to the work of Looijenga-Lunts \cite{LL97} and Verbitsky \cite{Ver96}.~In doing so we hope that this picture will give the right generalization of Esnault-Srinivas's result (Theorem \ref{Esnault-Srinivas},~(2)) to higher dimensional varieties over $\C$.

In the language of Gromov algebra,~the proof of Esnault-Srinivas (over $\C$) uses polarized Hodge structure on $H^{2}(X,\Q)$ \cite[Proposition 5.1]{Esnault-Srinivas} to obtain restrictions on eigenvalues of an automorphism of a surfaces on the orthogonal complement (with respect to cup-product) of the Gromov algebra with respect to any ample class.~A natural question then is how do we decompose the cohomology of a higher dimensional variety into $f^*$-stable subspaces such that the Gromov algebra is a `natural' component? 

We now give a plausible answer to this question.

\subsection{The LLV Lie algebra}

Now we summarize the key construction from \cite{LL97} and \cite{Ver96}.

Let $K$ be a field of characteristic $0$.~Let $M_{\bullet}$ be a $\Z$-graded finite dimensional $K$-vector space.~Let $h \in \text{End}(M)$ be such that $h$ acts by $i$ in degree $i$.~Thus the eigenspaces of $h$ determine the grading of $M_{\bullet}$,~and $u \in \text{End}(M)$ has degree $i$ iff $[h,u]=iu$.

An endomorphism $e \in \text{End}(M)$ of degree $2$ is said to have the \textit{Lefschetz property},~if for every $i \in \Z,~e^{i} \colon M_{-i} \to M_{i}$ is an isomorphism.~Equivalently by the Jacobson-Morozov lemma this is equivalent to existence of a linear transformation $f \in \text{End}(M)$ such that $f$ is of degree $-2$ such that $[e,f]=h$.

Let $\la$ be a finite dimensional $K$-vector space.~We regard $\la$ as a graded abelian Lie algebra which is homogeneous of degree $2$.~We say that a graded Lie homomorphism $e \colon \la \to \gl(M)$ has the \textit{Lefschetz property} if for some $a \in \la$,~$e(a)$ has the Lefschetz property.~Clearly having Lefschetz property is a Zariski open condition on $\gl(M)$,~and by Jacobson-Morozov we have a rational morphism $f \colon \la \to \gl(M)$ defined on this open subset.~The Lie sub-algebra of $\gl(M)$ generated by $e(a), f(a)$ for all possible choices of $a$ (denoted by $\g(\la,M)$) is the LLV lie algebra.~Note that $\text{ad}(h)$ induces a grading on $\g(\la,M)$ and that $\g(\la,M)$ is evenly graded under this action.

\begin{defn}\label{L_Mod}[Lefschetz Modules]
We say that the pair $(\la,M)$ is a \textit{Lefschetz $\la$-Module} if $\g(\la,M)$ is semisimple.
\end{defn}

\subsubsection*{Basic properties of the Lefschetz Module}

In what follows we fix $\la$ as above.

\begin{enumerate}[(i)]

\item The collection of Lefschetz $\la$-modules is closed under direct sums,~taking tensor products and taking duals.

\item The category of Lefschetz $\la$-modules is a semi-simple,~Artinian and Noetherian category. 

\item Given a Lefschetz module $M$,~any representation of $\g(\la,M)$ is also a Lefshcetz module (under the natural $\la$-action).~Moreover this correspondence preserves irreducibility. 

\item For any Lefschetz module $M$,~$\g(\la,M)$ is naturally graded and compatible with the action on $M$.

\end{enumerate}

\subsection{Connections with the Gromov algebra}
Let $X$ be a smooth projective variety over $\C$ of dimension $r$.~Let $f \colon X \to X$ be a finite self-map.~Then $f^*$ is an isomorphism on the top Betti cohomology and hence by non-degeneracy of the cup-product pairing it is so on all of $H^*(X,\Q)$.~Let $\omega \in H^2(X,\Q)$ be an ample class.

Let $\la$ be the $f^*$ stable subspace generated by $\omega$,~considered as a graded Lie algebra in degree 2.~Note that $\la$ is by definition the degree $2$ summand of $\text{A}^{Gr}_{\text{hom}}(f,\omega)_{\Q}$.~Let $M \colonequals H^{*}(X,\Q)[r]$ be the shifted (by $r$) total Betti cohomology of $X$.~Note that there is a natural graded map from $\la \to \gl(M)$ given by cup product,~and we denote the associated Lie algebra $\g(\la,M)$ by $\g(f,\omega)$.~Following result is an immediate consequence of \cite[Proposition 1.6]{LL97}.

\begin{prop}\label{semi_simple}
$\g(f,\omega)$ is semi-simple and hence $H^*(X,\Q)[r]$ is a Lefschetz module over the degree $2$ summand of $\text{A}^{Gr}_{\text{hom}}(f,\omega)_{\Q}$.
\end{prop}

\begin{rmk}\label{base_extension}
To be precise \cite[Proposition 1.6]{LL97} would show that $\g(f,\omega)\otimes_{\Q}\C$ is semi-simple,~which then implies the same for $\g(f,\omega)$ since non-degeneracy of the Killing form can be checked after an extension of base field.
\end{rmk}

We end this brief section by showing that $\text{A}^{Gr}_{\text{hom}}(f,\omega)_{\Q}$ is an irreducible summand of $H^*(X,\Q)[r]$ under the action of $\g(f,\omega)$.

\begin{prop}\label{Gromov_Lefschetz}
The homological Gromov algebra $\text{A}^{Gr}_{\text{hom}}(f,\omega)_{\Q}[r]$ is an irreducible representation of $\g(f,\omega)$ and hence equivalently also an irreducible Lefschetz module.
\end{prop}

\begin{proof}
We will at once show that $\text{A}^{Gr}_{\text{hom}}(f,\omega)_{\Q}[r]$ is a representation of $\g(f,\omega)$ and that it is an irreducible one.~First note that the lowest graded piece of $\text{A}^{Gr}_{\text{hom}}(f,\omega)_{\Q}[r]$ lies in the kernel of $\g(f,\omega)_{<0}$,~the negatively graded part of $\g(f,\omega)$.~Thus \cite[Proposition 1.12]{LL97} implies that the $\g(f,\omega)$ stable subspace generated by the lowest graded piece of $\text{A}^{Gr}_{\text{hom}}(f,\omega)_{\Q}[r]$ is generated by the action of Lefschetz operators and hence is equal to $\text{A}^{Gr}_{\text{hom}}(f,\omega)_{\Q}[r]$.~Irreducibility is a consequence of \cite[Corollary 1.13]{LL97},~since the lowest graded piece of $\text{A}^{Gr}_{\text{hom}}(f,\omega)_{\Q}[r]$ is a one dimensional space.
\end{proof}

Proposition \ref{Gromov_Lefschetz} seems to suggest that the Gromov algebra is one piece of a natural decomposition of the cohomology into $f^*$ stable subspaces.~Thus it would be interesting to study constraints analogous to Theorem \ref{Esnault-Srinivas} and Theorem \ref{Gromov_arbitrary_base_field} on the other pieces of this decomposition.~A natural starting point would be the cohomology of compact hyper-K\"ahlerian manifolds,~where both the LLV lie algebra and questions of entropy are rather well studied (see for example \cite{GLR19},~\cite{Obe19},~\cite{Ogu07}).~We hope to come back to this question in the future.

\bibliography{your_bib_archive}

\end{document}